\documentclass[8pt, reqno, oneside]{amsart}
\usepackage{amsmath,amssymb,amsfonts,amscd}

\newtheorem{definition}{Definition}
\newtheorem{theorem}[definition]{Theorem}
\newtheorem{lemma}[definition]{Lemma}
\newtheorem{corollary}[definition]{Corollary}

\newtheorem{remarks}[definition]{Remarks}

\newcommand{\from}{\colon}
\newcommand{\Aut}{\operatorname{Aut}}

\numberwithin{figure}{subsection}

\usepackage{graphicx}
\usepackage{amsmath}
\usepackage{latexsym,graphicx,graphics,framed}
\usepackage{amsfonts}
\usepackage{amssymb}
\usepackage{epsfig}
\usepackage{verbatim}
\usepackage{setspace}

\usepackage{amsmath,amssymb,amsfonts,amscd,pictexwd,dcpic}

\begin{document}
\title[Coverings of Profinite Graphs]{Coverings of Profinite Graphs}

\author{Amrita Acharyya}
\address{Department of Mathematics and Statistics\\
University of Toledo, Main Campus\\
Toledo, OH 43606-3390}
\email{Amrita.Acharyya@utoledo.edu}

\author{Jon M. Corson}
\address{Department of Mathematics\\
University of Alabama\\
Tuscaloosa, AL 35487-0350}
\email{jcorson@ua.edu}

\author{Bikash Das}
\address{Department of Mathematics\\
University of North Georgia, Gainesville Campus\\
Oakwood, Ga. 30566}
\email{Bikash.Das@ung.edu}
\doublespacing
\begin{abstract}

We define a covering of a profinite graph to be a projective limit of a system of covering maps of finite graphs. With this notion of covering, we develop a covering theory for profinite graphs which is in many ways analogous to the classical theory of coverings of abstract graphs. For example, it makes sense to talk about the universal cover of a profinite graph and we show that it always exists and is unique. We define the profinite fundamental group of a profinite graph and show that a connected cover of a connected profinite graph is the universal cover if and only if its profinite fundamental group is trivial.
\end{abstract}
\maketitle
\section{Introduction}

In this work we develop the idea of coverings for abstract graphs and the topological covering space theory in the category of Profinite Graphs which is in many senses similar to the idea of a Galois Cover, first studied and coined by Pavel Zalesskii~\cite{pZ89}. We define a more general covering of profinite graphs to be a projective limit of a system of covering maps of finite graphs. With this notion of covering, we developed a covering theory for profinite graphs which is in many ways analogous to the classical theory of coverings of abstract graphs and topological covering spaces. We give a necessary and sufficient condition for a continuous map of profinite graphs to lift to a profinite covering graph. Our condition is a projective analogue of the well-known condition in the topological theory of covering spaces. We define the covering transformations in a similar way as in the classical theory and found that the group of covering transformations of a connected covering graph is a profinite group that acts freely and uniformly equicontinuously on the covering graph.

 In light of our definition of a covering of profinite graphs we define a regular covering of profinite graphs to be a projective limit of regular coverings of finite path connected graphs. At this point we show that our regular cover turns to be what Zalesskii has called a Galois cover. We prove that the action of the group of covering transformation of a regular covering graph is continuous, residually free and is simply transitive on each vertex fiber. Conversely, we show that if a profinite group acts continuously and residually freely on a profinite connected graph $\Gamma$, without edge inversions, then the corresponding orbit mapping is a regular covering map of profinite graphs; moreover the profinite group turns out to be the group of covering transformations of the corresponding covering graph. We also give a criterian for a map of profinite graphs to be a regular covering in terms of Good Pairs of compatible cofinite entourages. We show that, any closed subgroup $H$ of $G$, the group of covering transformations of a regular covering graph, gives rise to a profinite covering graph of the corresponding base graph, and this is again a regular covering if $H$ is a normal subgroup of $G$. Moreover, the original regular covering graph is a regular covering arising in that way with $H=G$. 

We define the profinite analogue of fundamental group first in the special case of a finite discrete graph. Then we extended this idea to an arbitrary profinite graph by taking the projective limit of the profinite fundamental groups of its finite discrete quotient graphs. We notice that for any map of profinite graphs there exists a natural induced homomorphism of the profinite analogue of the fundamental groups. 

 We define a universal covering of profinite connected graphs by way of a universal mapping property as in the case of universal coverings of abstract connected graphs and path connected topological spaces. We first notice that if the universal cover of a profinite connected graph exists then it is unique up to isomorphism of profinite covering graphs. We give a construction of the universal profinite covering graph for any connected profinite graph. Then, by examining the graphs arising from this construction, we get a better understanding of the properties of universal profinite covering graphs. In particular, we show that a universal covering of profinite graphs is a regular covering. From our construction we also notice that a connected covering graph is the universal covering of the base graph if and only the profinite analogue of the fundamental group of that covering graph based at any of its vertices is trivial. Applying results of regular coverings and good pairs to the universal covering of a connected profinite graph $\Delta$ leads to a characterization of all profinite covering graphs of $\Delta$ in terms of closed subgroups of the group of covering transformations of the universal profinite covering graph of $\Delta$.


\section{Definition and elementary properties}\label{s:Definition}

For more details about cofinite graphs and profinite graphs see ~\cite{ACD13}.

By a covering of profinite graphs, we will mean a projective limit of an inverse system of locally bijective maps of finite discrete graphs. To make this precise, we first discuss inverse systems of maps of cofinite graphs in general. 

Let $I$ be a directed set. An {\it inverse system of maps of cofinite graphs\/} 
$(f_i\from\Gamma_i\to\Delta_i,\phi_{ij},\psi_{ij})$
indexed by $I$ consists of:
\begin{enumerate}
\item[(i)] for each $i\in I$, a continuous map of cofinite graphs $f_i\from\Gamma_i\to\Delta_i$;
\item[(ii)] for each $j\ge i$, continuous maps of graphs $\phi_{ij}\from\Gamma_j\to\Gamma_i$ and $\psi_{ij}\from\Delta_j\to\Delta_i$ such that the diagram
$$\begin{CD}
\Gamma_j @>\phi_{ij}>> \Gamma_i \\
@Vf_jVV @VVf_iV \\
\Delta_j @>\psi_{ij}>> \Delta_i 
\end{CD}$$ 
commutes;
\item[(iii)] for all $k\ge j\ge i$, we require that $\phi_{ik}=\phi_{ij}\phi_{jk}$ and $\psi_{ik}=\psi_{ij}\psi_{jk}$, furthermore  
for all $i \in I, \phi_{ii} = id_{\Gamma_i}, \psi_{ii} = id_{\Delta_i}$.
\end{enumerate}

In other words, an inverse system of maps of cofinite graphs is an inverse system in the category whose objects are continuous maps of cofinite graphs and whose morphisms are pairs of continuous maps of graphs that form commutative diagrams. 

An {\it inverse limit\/} of the system $(f_i\from\Gamma_i\to\Delta_i,\phi_{ij},\psi_{ij})$ consists of: 
\begin{enumerate}
\item[(i)] a continuous map of cofinite graphs $f\from\Gamma\to\Delta$ and two families of continuous maps of cofinite graphs $(\phi_i\from\Gamma\to\Gamma_i)_{i\in I}$ and $(\psi_i\from\Delta\to\Delta_i)_{i\in I}$ such that: for each $i\in I$, the diagram
$$\begin{CD}
\Gamma @>\phi_i>> \Gamma_i \\
@VfVV @VVf_iV \\
\Delta @>\psi_i>> \Delta_i 
\end{CD}$$ 
commutes, and for each $j\ge i$, $\phi_i=\phi_{ij}\phi_j$ and $\psi_i=\psi_{ij}\psi_j$.
\item[(ii)] If for another continuous map of cofinite graphs $f^{\prime}\from\Gamma^{\prime}\to\Delta^{\prime}$ and two families of continuous maps of cofinite graphs $(\phi_i^{\prime}\from\Gamma^{\prime}\to\Gamma_i)_{i\in I}$ and $(\psi_i^{\prime}\from\Delta^{\prime}\to\Delta_i)_{i\in I}$ such that, for each $i\in I$, the diagram
$$\begin{CD}
\Gamma^{\prime} @>\phi_i^{\prime}>> \Gamma_i \\
@Vf^{\prime}VV @VVf_iV \\
\Delta^{\prime} @>\psi_i^{\prime}>> \Delta_i 
\end{CD}$$ 
commutes, and for each $j\ge i$, $\phi_i^{\prime}=\phi_{ij}\phi_j^{\prime}$ and $\psi_i^{\prime}=\psi_{ij}\psi_j^{\prime}$, then there exist unique continuous maps of graphs 

$$
h\from\Gamma^{\prime} \to \Gamma, g\from\Delta^{\prime}\to \Delta
$$ 

such that the following diagram commutes.

$$
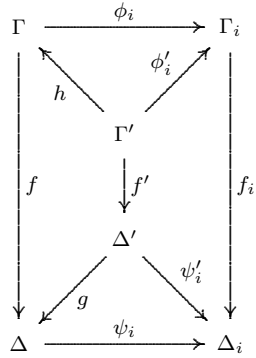
\begindc{\commdiag}[20]

\obj(0,0)[1]{$\Delta$}
\obj(2,2)[2]{$\Delta^{\prime}$}
\obj(4,0)[3]{$\Delta_i$}
\obj(0,6)[4]{$\Gamma$}
\obj(2,4)[5]{$\Gamma^{\prime}$}
\obj(4,6)[6]{$\Gamma_i$}
\mor{1}{3}{$\psi_i$}
\mor{2}{1}{$g$}
\mor{2}{3}{$\psi_{i}^{\prime}$}
\mor{4}{1}{$f$}
\mor{5}{2}{$f^{\prime}$}
\mor{6}{3}{$f_i$}
\mor{4}{6}{$\phi_i$}
\mor{5}{4}{$h$}
\mor{5}{6}{$\phi_i^{\prime}$}
\enddc$$

\end{enumerate}

If the inverse limit exists, then it is unique up to a natural isomorphism of graphs. 

To see that such inverse limits exists, we can construct the inverse limit of an inverse system 
$(f_i\from\Gamma_i\to\Delta_i,\phi_{ij},\psi_{ij})$ of maps of cofinite graphs as follows:

Let $\Gamma=\varprojlim(\Gamma_i,\phi_{ij})$ and
$\Delta=\varprojlim(\Delta_i,\psi_{ij})$; denote the canonical projections by  $\phi_i\from\Gamma\to\Gamma_i$ and $\psi_i\from\Delta\to\Delta_i$. 
Then the compositions $(f_i\circ\phi_i\from\Gamma\to\Delta_i)_{i\in I}$ form a compatible family of continuous maps of graphs and thus induce a continuous map of graphs $f\from\Gamma\to\Delta$. 
It can be shown that these maps satisfy the requirements of the inverse limit. 

Since the inverse limit is unique up to isomorphism of cofinite graphs, we write 
$$
f=\varprojlim(f_i\from\Gamma_i\to\Delta_i,\phi_{ij},\psi_{ij})
$$ 
or 
$f=\varprojlim f_i$ when the system of maps is understood. Our definition of a covering can now be stated precisely as follows.

\begin{definition}[Covering]\rm \label{d:Covering}
A map of profinite graphs $f\from\Gamma\to\Delta$ is a {\it covering\/} if it is the inverse limit of an inverse system of maps of graphs 
$(f_i\from\Gamma_i\to\Delta_i,\phi_{ij},\psi_{ij})$, where each $f_i\from\Gamma_i\to\Delta_i$ is a locally bijective map of finite discrete graphs.
In this situation, we call the pair $(\Gamma,f)$ a {\it profinite covering graph\/} of $\Delta$.
\end{definition}

For the rest of this subsection, let $I$ be a directed set and let $(f_i\from\Gamma_i\to\Delta_i,\phi_{ij},\psi_{ij})$ be an  inverse system of locally bijective maps of finite discrete graphs, indexed by $I$. Then the inverse limit $f\from\Gamma\to\Delta$ of this system is a covering map of profinite graphs. Denote the canonical projections by $\phi_i\from\Gamma\to\Gamma_i$ and $\psi_i\from\Delta\to\Delta_i$.

\begin{remarks}
\begin{enumerate}
\item  $\varprojlim\Gamma_i = \Gamma = \varprojlim\phi_i(\Gamma), \varprojlim\Delta_i = \Delta = \varprojlim\psi_i(\Delta)$
\item $\{R_i\mid i\in I\}$, $\{S_i\mid i\in I\}$ form fundamental systems of compatible cofinite entourages for $\Gamma$, $\Delta$ respectively where $R_i=\phi_i^{-1}\phi_i =$ kernel of $\phi_{i} = \{(x,y)\in \Gamma \mid \phi_i(x) = \phi_i(y)\} , S_i=\psi_i^{-1}\psi_i =$ kernel of $\psi_{i} = \{(x,y)\in \Delta \mid \psi_i(x) = \psi_i(y)\}$, for each $i\in I$. 
\item We may assume that the canonical maps $\psi_i\from\Delta\to\Delta_i$ are surjective, and thus $\Delta_i = \psi_i(\Delta)=\Delta/S_i$ and $\Gamma_i = f_{i}^{-1}[(\psi_{i}(\Delta)] = f_{i}^{-1}(\Delta_i)$. 
\item If both $\Delta$ and $\Gamma$ are connected, then we may assume that each $\Delta_i$ and each $\Gamma_i$ is path connected.
\end{enumerate}
\end{remarks}
\begin{proof}
We will now provide brief proofs of the above remarks. We frequently use a few facts about inverse limits of profinite graphs listed as Mathematical Preliminaries.
\begin{enumerate}
\item Since $\Gamma$ is closed, in itself $\varprojlim\phi_i(\Gamma) = \Gamma$ and similarly, $\Delta = \varprojlim\psi_i(\Delta)$.
\item $R_i=\phi_i^{-1}\phi_i$, $S_i=\psi_i^{-1}\psi_i$, are compatible cofinite entourages over $\Gamma$ and $\Delta$ respectively, for each $i\in I$. Without loss of generality one can take $\Gamma/R_i = \phi_i(\Gamma)$. Thus we can successfully claim that $\varprojlim\Gamma/R_i = \Gamma$ and $\varprojlim\Delta/S_i = \Delta$. 

Now let $R$ be any compatible cofinite entourage over $\Gamma$ and let $\eta_R\from\Gamma\to\Gamma/R$ be the natural quotient map. There exists some $R_i$ such that the following diagram commutes for some continuous map of graphs $\phi_{RR_i}\from\Gamma/R_i\to\Gamma/R$.

$$\begindc{\commdiag}[30] 
\obj(0,0)[1]{$\Gamma$}
\obj(0,-2)[2]{$\Gamma/R_i$}
\obj(2,-2)[3]{$\Gamma/R$}
\mor{1}{2}{$\phi_i$}
\mor{1}{3}{$\eta_R$}
\mor{2}{3}{$\phi_{RR_i}$}
\enddc$$

Hence $R_i\subseteq R$. Thus $\{R_i\mid i\in I\}$ forms a fundamental system of compatible cofinite entourages over $\Gamma$. Similarly $\{S_i\mid i\in I\}$ forms a fundamental systems of compatible cofinite entourages over $\Delta$.
\item Let us denote $f_{i}|_{f_{i}^{-1}[\psi_{i}(\Delta)]}$ by $g_i$, $\psi_{i}(\Delta)$ by $\Delta_{i}^{\prime}$ and $f_{i}^{-1}[(\psi_{i}(\Delta)]$ by $\Gamma_i^{\prime}$. Thus  \hbox{$f_{i}|_{f_{i}^{-1}[(\psi_{i}(\Delta)]}\from f_{i}^{-1}[(\psi_{i}(\Delta)]\to \psi_{i}(\Delta)$}  will now be viewed as \hbox{$g_i\from\Gamma_i^{\prime}\to\Delta_i^{\prime}$}.
 It turns out that each $g_{i}$ is locally bijective and $(\Gamma_i^{\prime}, \phi_{ij})_{i,j \in I, j\geq i}$ forms an inverse system of finite discrete graphs (by a little abuse of notation, we assume $\phi_{ij} = \phi_{ij}|_{\Gamma_j^{\prime}}$), where $\phi_{i}(\Gamma) \subseteq \Gamma_i^{\prime} \subseteq \Gamma_{i},$ for all $i$ in $I$.
 Thus $\Gamma = \varprojlim \phi_{i}(\Gamma) = \varprojlim \Gamma_i^{\prime} = \varprojlim \Gamma_{i}$. So without loss of generality, we are able to replace $\Delta_{i}$ by  $\Delta_i^{\prime} = \psi_{i}(\Delta) = \Delta/S_i$ and $\Gamma_{i}$ by $\Gamma_i^{\prime} = f_{i}^{-1}((\psi_{i}(\Delta))$.
Hence, if $\Delta$ is connected, then each $\Delta/S_i$ is path connected and we may assume that each $\Delta_i$ is path connected.
\item Let us choose a vertex $v = (v_{i})_{i \in I}$ in $\Gamma = \varprojlim\Gamma_i$ where, without loss of generality, we may assume $(f_{i}^{-1}(\Delta_i)) = \Gamma_i$ and $\psi_{i}(\Delta) = \Delta_i$. Since $\Delta$ is cofinitely connected, each $\Delta_i$ is path connected. Let us choose $\Gamma_{i}^{\prime}$ as the path component of $\Gamma_i$ containing $v_{i}$.  It follows that $\phi_i(\Gamma)\subseteq \Gamma_i^{\prime}\subseteq \Gamma_{i}$. 
So, $\Gamma = \varprojlim \phi_{i}(\Gamma) = \varprojlim (\Gamma_{i}^{\prime}) = \varprojlim \Gamma_{i} = \Gamma$. 

Thus without loss of generality we may assume that each $\Gamma_i = \Gamma_i^{\prime}$ is path connected.
\end{enumerate}
\end{proof} 

In the following lemma, we list some elementary properties of coverings. 

\begin{lemma}\label{elem prop}
Let $f\from\Gamma\to\Delta$ be a covering of profinite graphs. Then the following properties hold:
\begin{enumerate}
\item[(a)] $f$ is a locally bijective map of graphs;
\item[(b)] if $\Delta$ is connected, then $f$ is surjective and thus $f$ is a quotient map; 
\item[(c)] For each edge $e\in\Delta$, we define a map $a\mapsto a.e$ from $f^{-1}(s(e))$ to $f^{-1}(t(e))$ by requiring that $s(\tilde{e}).e = t(\tilde{e})$ for all $\tilde{e}$ in $f^{-1}(e)$.
Then $a\mapsto a.e$ is an isomorphism of uniform spaces.
\end{enumerate}
\end{lemma}
\begin{proof}
We have $f = \varprojlim(f_{i}:\Gamma_{i}\rightarrow\Delta_{i},\phi_{ij}, \psi_{ij})$, where each $f_i$ is locally bijective.
\begin{enumerate}
\item[(a)] Since for each $i\in I,  f_i$ is  locally bijective it follows that $f$ is also so.
\item[(b)] By remark~(4) each $\Delta_i$ is path connected. 
Since $(\Gamma_i,f_i)$ is a (nonempty) covering graph of the path connected graph $\Delta_i$, it follows that each $f_i$ is surjective. 
Let $y\in\Delta$ be given and write $y_i=\psi_i(y)$ for $i\in I$.
Then $X_i=f_i^{-1}(y_i)$ is nonempty and compact for each $i\in I$. 
Hence $X=\varprojlim X_i$ is nonempty.
For any $x\in X, f(x)=y$.

Moreover since $f$ is a continuous map from a compact space $\Gamma$ to a Hausdorff space $\Delta$ it is a closed continuous surjection and hence a quotient map. 
\item[(c)] Let $\tau\from f^{-1}(s(e))\to f^{-1}(t(e))$ denote the mapping taking $a\mapsto a.e$ where $a.e = t(\tilde{e})$ whenever $s(\tilde e) = a$, for all $\tilde{e}$ in $f^{-1}(e)$.  $\tau$ is a well defined bijection with inverse $\eta\from f^{-1}(t(e))\to f^{-1}(s(e))$, the mapping taking $b\mapsto b.\overline e$. 

We use Remark~2 to prove this. Consider $R_{i}$ for some $i\in I$. Let, $a_1,a_2 \in f^{-1}(s(e))$ and $(a_1,a_2) \in R_i$. Then, $\phi_{i}(a_1).\psi_{i}(e) = \phi_{i}(a_2).\psi_{i}(e)$ where we consider the map $\tau_i \from f_{i}^{-1}(s(\psi_{i}(e)))\to f^{-1}(t(\psi_{i}(e)))$ in the corresponding finite graph level behaving exactly like $\tau$. Now if $\tilde{e_1}, \tilde{e_2} \in f^{-1}(e), a_1,a_2 \in f^{-1}(s(e))$, then $\phi_{i}(a_1).\psi_{i}(e) = t\phi_{i}(\tilde{e_1}) = \phi_{i}(t(\tilde{e_1})) =  \phi_{i}(a_{1}.e)$ and $\phi_{i}(a_2).\psi_{i}(e) = t\phi_{i}(\tilde{e_2}) = \phi_{i}(t(\tilde{e_2})) =  \phi_{i}(a_{2}.e)$. Thus it follows that $(a_{1}.e, a_{2}.e) \in R_i.$  So, $\tau$ is uniformly continuous. Hence the result follows. 
\end{enumerate}
\end{proof}


\section{General lifting criterion} \label{s:General lifting}

In this subsection we give a necessary and sufficient condition for a continuous map of profinite graphs to lift to a profinite covering graph. Our condition is a projective analogue of the well-known condition in the topological theory of covering spaces. As in that classical theory, we first observe that the usual uniqueness of lifts result holds for coverings of profinite groups.

Let $f\from\Gamma\to\Delta$ be a covering of profinite graphs and fix an inverse system $(f_i\from\Gamma_i\to\Delta_i,\phi_{ij},\psi_{ij})$ of locally bijective maps of finite discrete graphs, indexed by $I$, such that 
$f=\varprojlim f_i$. Denote the canonical projections by $\phi_i\from\Gamma\to\Gamma_i$ and $\psi_i\from\Delta\to\Delta_i$ and assume, as we may, that each $\psi_i$ is surjective and identify $\Delta_i$ with $\Delta/S_i$, where $S_i=\psi_i^{-1}\psi_i$, for each $i\in I$.

\begin{lemma} \label{l:Uniqueness of lifts}
Let $f\from\Gamma\to\Delta$ be a covering of profinite graphs and let $\Sigma$ be a connected profinite graph. If $h_1,h_2\from\Sigma\to\Gamma$ are continuous maps of graphs such that $fh_1=fh_2$, and if $h_1(c)=h_2(c)$ for some $c\in\Sigma$, then $h_1=h_2$.
\end{lemma}

\begin{proof}
Let us first notice a general fact about maps of profinite graphs and compatible cofinite entourages. 

If $f\from \Gamma\to\Delta$ is a map of profinite graphs and $R, S$ be two compatible cofinite entourages over $\Gamma, \Delta$ respectively such that $(f\times f)[R]\subseteq S$. Then $f_{SR}\from \Gamma/R\to\Delta/S$ defined via $f_{SR}(R[x]) = S[f(x)],$ for all $x\in\Gamma$ is a well defined map of graphs. 

Without loss of generality we may assume that $c \in V(\Sigma)$. Let  $h_1(c)=h_2(c)$. Let us take $x \in \Sigma$. Consider $h_1(x)$ and $h_2(x)$. Let $ i \in I$. Set

$$
T_{i} = (h_{1}\times h_{1})^{-1}(R_{i})\bigcap (h_{2}\times h_{2})^{-1}(R_{i})
$$
Then, $\Sigma /T_{i}$ is path connected. Let $g =fh_1 =fh_2$. Then, $(g\times g)[T_i] = (f\times f) ((h_{1}\times h_{1})[T_i]) \subseteq (f\times f)[R_i] \subseteq S_i$. So, now consider the natural maps of graphs $g_{S_{i}T_{i}} \from \Sigma/T_{i} \to \Delta/S_{i}$ and $f_{S_{i}R_{i}} \from \Gamma/R_{i} \to \Delta/S_{i}$. Also, we observe that $(h_{1}\times h_{1})[T_{i}] \subseteq R_i, (h_{2}\times h_{2})[T_{i}] \subseteq R_i $. So, again consider the natural maps of graphs  $ (h_{1})_{R_{i}T_{i}} \from \Sigma/T_i \to \Gamma/R_i$, and $(h_{2})_{R_{i}T_{i}} \from \Sigma/T_i \to \Gamma/R_i$. We note that $f_{S_{i}R_{i}}(h_{1})_{R_{i}T_{i}} = f_{S_{i}R_{i}}(h_{2})_{R_{i}T_{i}}$ and $(h_{1})_{R_{i}T_{i}}(T_{i}[c]) = (h_{2})_{R_{i}T_{i}}(T_{i}[c])$. Now we have $\Gamma/R_i$ is embedded inside $\Gamma_{i}$ as a subgraph, and so is $\Delta/S_i$ inside $\Delta_{i}$, where $f_i\from \Gamma_{i} \to \Delta_{i}$ is a local bijection of finite graphs and the restriction of $f_i$ over $\Gamma/R_{i}$ is $f_{S_{i}R_{i}}$. Hence by the results of abstract graph theory it follows that $(h_{1})_{R_{i}T_{i}} = (h_{2})_{R_{i}T_{i}}$ since $\Sigma/T_{i}$ is path connected and $f_{S_{i}R_{i}}$ is locally injective. Thus we have for any $x \in \Sigma, (h_{1}(x), h_{2}(x)) \in R_{i}$, for all $i \in I$. Thus , $h_{1}(x) = h_{2}(x)$, for all $x \in \Sigma$.  
\end{proof}

\begin{theorem}\label{Lifting Criterion}
Suppose $\Sigma$ is a connected profinite graph, $g\from\Sigma\to\Delta$ is a continuous map of graphs, and let $a$, $b$, $c$ be vertices  of $\Gamma$, $\Delta$, $\Sigma$ such that $f(a)=b=g(c)$. 
Then there exists a unique continuous map of graphs $h\from\Sigma\to\Gamma$ such that $h(c)=a$ and $fh=g$ if and only if for each $i\in I$, there exists a compatible cofinite entourage $T_i$ of $\Sigma$ such that 
$(g\times g)[T_i]\subseteq S_i$ and $g_{S_iT_i}\,\pi_1(\Sigma/T_i,T_i[c])\subseteq f_i\,\pi_1(\Gamma_i,a_i)$, where $g_{S_iT_i}\from (\Sigma/T_i,T_i[c]) \to (\Delta_i,b_i)$ is the natural map, with $\Delta_i = \Delta/S_i, b_i = \psi_i(b)$.

$$\begindc{\commdiag}[30]
\obj(0,0)[Sigma]{$\Sigma,c$}
\obj(2,0)[Delta]{$\Delta,b$}
\obj(2,2)[Gamma]{$\Gamma,a$}
\mor{Sigma}{Delta}{$g$}
\mor{Gamma}{Delta}{$f$}
\mor{Sigma}{Gamma}{$h$}[\atleft,\dashArrow]
\enddc$$
\end{theorem}

\begin{proof}
Let us first assume that for each $i\in I$, there exists a compatible cofinite entourage $T_i$ of $\Sigma$ such that 
$(g\times g)[T_i]\subseteq S_i$ and $g_{S_iT_i}\,\pi_1(\Sigma/T_i,T_i[c])\subseteq f_i\,\pi_1(\Gamma_i,a_i)$, where, by abuse of notation,  
$$
g_{S_iT_i}\from \,\pi_1(\Sigma/T_i,T_i[c]) \to \,\pi_1(\Delta_i,b_i)
$$ 
is the induced map of fundamental groups obtained from the map of abstract graphs $g_{S_iT_i}\from(\Sigma/T_i,T_i[c]) \to (\Delta_i,b_i)$ with $g_{S_iT_i}(T_{i}[x]) = S_{i}[g(x)] = \psi_{i}(g(x)).$  Let $a = (a_{i})_{i \in I} , b = (b_{i})_{i \in I}$. Then $f_{i}(a_{i}) = b_{i},$ for all $i\in I $. By the general lifting criterion of finite graphs there exists a unique lift $h^{i}_{T_{i}}\from (\Sigma/T_i,T_i[c]) \to (\Gamma_i,a_i)$ such that $f_{i}h^{i}_{T_{i}} = g_{S{i}T_{i}}$. Define $h^{i}\from \Sigma \to \Gamma_{i}$ by  $h^{i} = h^{i}_{T_{i}}q_{T_{i}}$ where $q_{T_{i}}\from\Sigma \to \Sigma/T_i$ is the natural quotient map. Then $h^{i}$ is a continuous map of graphs as $h^{i}_{T_{i}}$ is a continuous map of finite discrete graphs and $q_{T_{i}}$ is a quotient map of graphs  and hence both of them are continuous. 
$$\begindc{\commdiag}[30]
\obj(0,0)[1]{$(\Sigma, c)$}
\obj(6,0)[2]{$(\Delta, b)$}
\obj(10,0)[3]{$(\Delta/S_i, b_i)$}
\obj(6,4)[4]{$(\Gamma, a)$}
\obj(10,4)[5]{$(\Gamma_i, a_i)$}
\obj(0,-2)[6]{$(\Sigma/T_i, T_i[c])$}
\mor{1}{6}{$q_{T_i}$}
\mor{1}{4}{$h$}
\mor{4}{5}{$\phi_i$}
\mor{4}{2}{$f$}
\mor{2}{3}{$\psi_i$}
\mor{5}{3}{$f_i$}
\mor{6}{5}{$h^{i}_{T_{i}}$}
\mor{1}{5}{$h^{i}$}
\mor{6}{3}{$g_{S{i}T_{i}}$}
\enddc$$
 $h^{i}$ is independent of the choice of $T_{i}$. This follows, since $\Sigma$ is a profinitely connected graph and $f_{i}\from(\Gamma_{i}, a_{i})\to (\Delta_{i},b_{i})$ is a covering map of profnite graphs we can apply Lemma~\ref{l:Uniqueness of lifts},. So, now we define $h\from\Sigma \to \Gamma$ as $h(x) = (h^{i}(x))_{i \in I}$. Let us check that $h(x) \in \Gamma$. So, we have to show that for any $x \in \Sigma, i\leq j \in I, \phi_{ij}h^{j}(x) = h^{i}(x)$. It follows by Lemma~\ref{l:Uniqueness of lifts}, since 
$\phi_{ij}h^{j}(c) = h^{i}(c)$ and $f_{i}\phi_{ij}h^{j} = f_{i}h^{i}$.
Here $h$, being a continuous map of profinite graphs (as each $h^{i}$ is so), is also a uniformly continuous map of profinite graphs. Clearly, $fh =g$. Also, the choice of $h$ is unique by Lemma~\ref{l:Uniqueness of lifts}.

Conversely, let such $h$ exist. Let $i \in I$ and let $R_{i}, S_{i} $ be as in remark~2 to  Definition~\ref{d:Covering}. Then $T_{i} = (g\times g)^{-1}(S_{i})\cap (h\times h)^{-1}(R_{i})$ is a compatible cofinite entourage over $\Sigma$ By the lifting criterion of finite graph theory it follows that $g_{S_iT_i}\,\pi_1(\Sigma/T_i,T_i[c])\subseteq f_i\,\pi_1(\Gamma_i,a_i)$.

\end{proof}


\section{Covering transformations}

We begin by defining homomorphisms and isomorphisms of covering graphs of a given profinite graph in the usual way. Here, and elsewhere, let $\Delta$ be a fixed profinite graph.

\begin{definition}\rm
Let $(\Gamma_1,f_1)$ and $(\Gamma_2,f_2)$ be profinite covering graphs of $\Delta$. 
A {\it homomorphism\/} from $(\Gamma_1,f_1)$ to $(\Gamma_2,f_2)$ is a continuous map of graphs $h\from\Gamma_1\to\Gamma_2$ such that $f_2h=f_1$.
A homomorphism $h$ from $(\Gamma_1,f_1)$ to $(\Gamma_2,f_2)$ is called an 
{\it isomorphism\/} if there exists a homomorphism $g$ from $(\Gamma_2,f_2)$ to $(\Gamma_1,f_1)$ such that both $gh$ and $hg$ are identity maps. 
\end{definition}

An isomorphism $h$ of a profinite covering graph $(\Gamma,f)$ of $\Delta$ to itself is called an automorphism or {\it covering transformation\/}. 
The set of all covering transformations of $(\Gamma,f)$ is a group under composition of maps; we denote it by $\Aut(\Gamma,f)$. 

As an immediate 
consequence of Lemma~\ref{l:Uniqueness of lifts} we see that:
if $(\Gamma_1,f_1)$ and $(\Gamma_2,f_2)$ are connected profinite covering graphs of $\Delta$, and if $h_1$ and $h_2$ are homomorphisms from $(\Gamma_1,f_1)$ to $(\Gamma_2,f_2)$ such that $h_1(a)=h_2(a)$ for some $a\in\Gamma_1$, then $h_1=h_2$.
Applying this to covering transformations, we obtain the following. 

\begin{lemma}
If $f\from\Gamma\to\Delta$ is a covering of profinite graphs and $\Gamma$ is connected, then the group of covering transformations $\Aut(\Gamma,f)$ acts freely on $\Gamma$.
\end{lemma}

A group action of an abstract group $G$ on a uniform space $X$ is said to be {\it uniformly equicontinous\/} if the set of translations $\{x\mapsto g\cdot x\mid g\in G\}$ is a uniformly equicontinous family of functions from $X$ to itself, that is, if for each entourage $W$ of $X$, there is an entourage $V$ of $X$ such that 
$(g\times g)\cdot V\subseteq W$ for all $g\in G$.

\begin{lemma}\label{uniform equicont}
If $f\from\Gamma\to\Delta$ is a covering of profinite graphs and $\Gamma$ is connected, then the group $\Aut(\Gamma,f)$ acts uniformly equicontinuously  on~$\Gamma$.
\end{lemma}

\begin{proof}
Let $(f_i\from\Gamma_i\to\Delta_i,\phi_{ij},\psi_{ij})$ be an inverse system of locally bijective maps of finite discrete graphs such that $f=\varprojlim f_i$. Suppose that $I$ is the directed index set corresponding to this inverse system. Then for $i\in I$ and each $\alpha\in\Aut(\Gamma,f)$, the map $\phi_i\alpha$ is a lift of $\psi_if$ to the finite covering graph $(\Gamma_i,f_i)$ of $\Delta_i$.
$$\begindc{\commdiag}[30]
\obj(0,0)[G1]{$\Gamma$}
\obj(2,0)[D]{$\Delta$}
\obj(2,2)[G2]{$\Gamma$}
\obj(3,0)[Di]{$\Delta_i$}
\obj(3,2)[Gi]{$\Gamma_i$}
\mor{G1}{D}{$f$}
\mor{G2}{D}{$f$}
\mor{G1}{G2}{$\alpha$}[\atleft,\solidarrow][
\mor{D}{Di}{$\psi_i$}
\mor{G2}{Gi}{$\phi_i$}
\mor{Gi}{Di}{$f_i$}
\enddc$$
Since, by Lemma~\ref{l:Uniqueness of lifts}, two lifts of $\psi_if$ are equal if they agree at one point, it follows that the set of maps
$A_i=\{\phi_i\alpha\from\Gamma\to\Gamma_i\mid \alpha\in \Aut(\Gamma,f)\}$ is finite, since $\Gamma_i$ is finite. 
For each $g=\phi_i\alpha\from\Gamma\to\Gamma_i$ in $A_i$, let 
$T_g=g^{-1}g=\alpha^{-1}\phi_i^{-1}\phi_i\alpha=(\alpha\times\alpha)^{-1}[R_i]$. It follows that the (finite) intersubsection $T=\bigcap_{g\in A_i}T_g$ is a compatible cofinite entourage of $\Gamma$. 
Now, for each $\alpha\in\Aut(\Gamma,f)$, $\phi_i\alpha=g$ for some $g\in A_i$ and hence $(\alpha\times\alpha)[T]\subseteq(\alpha\times\alpha)[T_g]\subseteq R_i$. However, the set of all $R_i$, $i\in I$, is a fundamental system of entourages of $\Gamma$. Thus it follows that the action of $\Aut(\Gamma,f)$ on $\Gamma$ is uniformly equicontinous.
\end{proof}

Given a connected profinite covering graph $(\Gamma,f)$ of $\Delta$, the previous lemma shows that $\Aut(\Gamma,f)$ acts uniformly equicontinuously on $\Gamma$. Thus, by ~\cite{ACD13}, $\Aut(\Gamma,f)$, with the uniformity induced by its action on $\Gamma$, is a cofinite group. From here on, we will always endow the group of covering transformations of a connected profinite covering graph with its cofinite group structure arising in this way. Moreover, we next prove that even more is true.

For more details about groups acting on cofinite and profinite graphs see ~\cite{ACD13}.

\begin{lemma}
If $f\from\Gamma\to\Delta$ is a covering of profinite graphs and $\Gamma$ is connected, then $\Aut(\Gamma,f)$ is a profinite group.

\begin{proof}
Since $G = \Aut(\Gamma,f)$ acts uniformly equicontinuously  on~$\Gamma$. Hence, $\widehat{G}$, the profinite completion of $G$ with respect to the induced separating filter base of cofinite congruences obtained from the action of $G$ over $\Gamma$, also acts over $\Gamma$ uniformly equicontinuously. In order to show that $G$ is compact we wish to show that $\widehat{G} \subseteq G$. Let $\alpha \in \widehat{G}$. Then since the action of $\widehat{G}$ over $\Gamma$ preserves the graph structure of $\Gamma, \alpha:\Gamma\rightarrow \Gamma$ is a map of graphs. Also $\alpha$ is continuous as the group action of $\widehat{G}$ over $\Gamma$ is uniformly continuous, by ~\cite{ACD13}. Also $\alpha^{-1} \in \widehat{G}$, as $\widehat{G}$ is a group. So, $\alpha:\Gamma\rightarrow \Gamma$ is a continuous map of graphs with continuous inverse. 

Now in order to show $\alpha \in G$ we have to prove $f\alpha = f$. Note that $\alpha = (N_{R}[\alpha_{R}])_{R\in J}$ where $J$ is the fundamental system of $G$-invariant compatible cofinite entourages over $\Gamma, N_R$ is the cofinite congruence over $G$ corresponding to the compatible cofinite entourage $R$, due to the group action $G\times\Gamma\to\Gamma$, and $\alpha_R\in G, $ for all $R\in J$. So, for any $i \in I$, there exists a $G$-invariant compatible cofinite entourage $R \subseteq R_{i}$. It follows that for all $i \in I$, and for all $x \in \Gamma$, since $f\circ\alpha_R = f$  

it follows that $f\circ\alpha = f$.
\end{proof}
\end{lemma}

Combining the above lemmas, we have the following theorem.

\begin{theorem}
Let $f\from\Gamma\to\Delta$ be a covering of profinite graphs, where $\Gamma$ is connected. 
Then the group of covering transformations $\Aut(\Gamma,f)$ is a profinite group that acts freely and uniformly equicontinuously on~$\Gamma$. 
\end{theorem}


\section{Regular coverings}

In light of our definition of a covering of profinite graphs (Definition~\ref{d:Covering}), it seems only natural to define a regular covering of profinite graphs to be a projective limit of regular coverings of finite path connected graphs. 
This is what we do in this subsection. However, in the next subsection we will see that there is an equivalent (more intrinsic) way to describe regular coverings. 

\begin{definition}[Regular covering] \label{d:Regular covering} \rm 
A map $f\from\Gamma\to\Delta$ of profinite graphs is a {\it regular covering\/} if it is the inverse limit of an inverse system of maps of graphs 
$(f_i\from\Gamma_i\to\Delta_i,\phi_{ij},\psi_{ij}), i,j \in I, j\geq i$ where each $f_i\from\Gamma_i\to\Delta_i$ is a regular covering of path connected finite discrete graphs.
In this situation, we call the pair $(\Gamma,f)$ a {\it regular profinite covering graph\/} of~$\Delta$.
\end{definition}

Throughout this section let $f\from \Gamma \to \Delta$ be a regular covering of profinite graph which is the inverse limit of an inverse system of maps of graphs 
$(f_i\from\Gamma_i\to\Delta_i,\phi_{ij},\psi_{ij}), i,j \in I, j\geq i$ where each $f_i\from\Gamma_i\to\Delta_i$ is a regular covering of path connected finite discrete graphs. We know that in this situation, all $\Gamma_{i}, \Delta_{i}$ are path connected as in remark~4 in section~\ref{s:Definition}. Also we will assume everywhere in this section that $G = \Aut(\Gamma,f)$, 
the group of covering transformations of the regular profinite covering graph $\Gamma$ corresponding to f.

We first show that, just as for regular coverings of path connected graphs and spaces in the classical sense, the group of covering transformations of a regular profinite covering graph acts transitively on each fiber.

\begin{lemma}
If $a_1$, $a_2$ are vertices of $\Gamma$ such that $f(a_1)=f(a_2)$, then there exists a unique covering transformation $\alpha$ of $(\Gamma,f)$ such that $\alpha(a_1)=a_2$.
\end{lemma}

\begin{proof}
Let $f(a_1)=f(a_2) = b = (b_{i})_{i \in I}$ and $a_1 = (a_{1,i})_{i \in I}, a_2 = (a_{2,i})_{i \in I}$. Since, $f_i$ restricted to $\Gamma/R_i$ can be viewed as the natural map of graphs 
$f_{S_{i}R_{i}}\from \Gamma/R_i \to \Delta/S_i$ and $f_{i}\from(\Gamma_{i},a_{1,i}) \to (\Delta_{i}, b_{i}), f_{i}\from(\Gamma_{i},a_{2,i}) \to (\Delta_{i}, b_{i})$ are regular covering of finite graphs by Theorem~\ref{Lifting Criterion} and Lemma~\ref{l:Uniqueness of lifts} the result follows.
\end{proof}

\begin{lemma}
Given any $h \in G , h = \varprojlim h_{i}$ where $h_{i}$ is a covering transformation of the finite covering graph $(\Gamma_{i}, f_{i})$ and $\{R_{i} \mid i \in I\}$ is a fundamental system of $G$-invarient compatible cofinite entourages over $\Gamma$.

\end{lemma}
\begin{proof}
Consider $a \in V(\Gamma)$ and let $h(a) = c.$  Then $f(c)= fh(a) = f(a) = b,$ where $a = (a_{i})_{i\in I}, b = (b_{i})_{i\in I}, c = (c_{i})_{i\in I}$. By the lifting criteria of Path Connected Coverings, for each $i \in I$, 
there exists a unique bijective map of graphs $h_i\from (\Gamma_i, a_i)\to (\Gamma_i, c_i)$ such that $f_{i}h_{i} = f_i$. Define $\alpha\from \Gamma \to \Gamma$ as $\alpha(x) = (h_{i}(x_{i}))_{i \in I}$ where $x = (x_{i})_{i \in I}$ . 
   By  Lemma~\ref{l:Uniqueness of lifts}, $\alpha$ is a well defined continuous map of graphs as each $h_i$ is so.  It can be shown that $f\alpha = f$.  By the uniqueness part in Lemma~\ref{l:Uniqueness of lifts} we have $\alpha = h$. Again, let $x = (x_i)_{i\in I}, y = (y_i)_{i\in I}$ and $(x,y) \in R_{i}$. So, $h_{i}(x_{i}) = h_{i}(y_{i})$. This implies $\phi_{i}h(x) = \phi_{i}h(y)$. So, $(h(x), h(y)) \in R_{i}$, for all $h\in G$. Hence the lemma follows.

\end{proof}

\begin{lemma}\label{resi free act}
Let $f\from\Gamma\to\Delta$ be a regular covering of profinite graphs. Then $\Gamma$ has a fundamental system of entourages consisting of $G$-invariant compatible cofinite equivalence relations $R$ such that the induced faithful action of $G/N_R$ on $\Gamma/R$ is free, where $N_R$ is the kernel of the induced action of $G$ on $\Gamma/R$.
\end{lemma}
\begin{proof}
$f$ is a covering of a profinite connected graph by Lemma~\ref{uniform equicont}, so $G$ acts uniformly equicontinuously over $\Gamma$. Hence $\Gamma$ has a fundamental system of $G$-invariant compatible cofinite entourages.

Consider the group action of $G/N_{R_{i}}$ over $\Gamma/R_{i}$ which is induced by the action of $G$ over $\Gamma$. For details see ~\cite{ACD13}. Let $N_{R_{i}}[h]R_{i}[x] = R_{i}[x]$ for some $h \in G$ and some $R_{i}[x] \in \Gamma/R_i$. So, $R_{i}[hx] = R_{i}[x]$. This means $x_i = R_{i}[x] = R_{i}[hx] = h_i(x_i)$, where $x = (x_{i})_{i\in I}, h = \varprojlim h_{i}$ as in the previous lemma. Since $h_{i}$ fixes one element of $\Gamma_i, h_{i}(y_i) = y_i$, for all $y_i \in \Gamma_i$, as $f_i\circ h_i = h_i$. This means $(hy,y) \in R_i$ for all $y \in \Gamma$. This implies $(h, 1_G) \in N_{R_{i}}$. Thus $N_{R_{i}}[h] = N_{R_i}[1_G]$, the identity element of $G/N_{R_{i}}$. Hence the action of $G/N_{R_{i}}$ over $\Gamma/R_{i}$ is free. 
\end{proof}

For convenience, we give a  name to the property that we just saw holds for the action of the group of covering transformations of a regular profinite covering graph. 

\begin{definition}\label{res free def} \rm
We say that an action of a group $G$ on a profinite graph $\Gamma$ is {\it residually free\/} if there exists a fundamental system $I$ of $G$-invariant compatible cofinite entourages of $\Gamma$ such that for each $R\in I$, the induced faithful action of $G/N_R$ on $\Gamma/R$ is free, where $N_R$ is the kernel of the action of $G$ on $\Gamma/R$.
\end{definition}

The following theorem uses this terminology and summarizes what we have shown so far for regular coverings.

\begin{theorem}\label{res free}
Let $f\from\Gamma\to\Delta$ be a regular covering of profinite graphs. Then $G$ is a profinite group whose action on $\Gamma$ is continuous, residually free, and simply transitive on the fiber $f^{-1}(b)$, for any vertex $b$ of $\Delta$. 
\end{theorem}

In the next subsection we prove that, conversely, every continuous residually free group action of a profinite group on a connected profinite graph gives rise to a regular covering.


\section{Good pairs and residually free group actions}

\begin{definition} \rm
Let $f\from\Gamma\to\Delta$ be a map of profinite graphs and let $R$, $S$ be compatible cofinite entourages of $\Gamma$, $\Delta$. We call $(R,S)$ a {\it good pair\/} for $f$ if 
$(f\times f)[R]\subseteq S$ and the induced map $f_{SR}\from\Gamma/R\to\Delta/S$ is locally bijective. If in addition, $f_{SR}$ is a regular covering of finite path connected graphs, then we say that the good pair $(R,S)$ is {\it regular}. 
\end{definition}

A family $I$ of good pairs for a map $f\from\Gamma\to\Delta$ of profinite graphs is called a {\it fundamental system\/} of good pairs if for all entourages $V$, $W$ of $\Gamma$, $\Delta$, there exists $(R,S)\in I$ such that $(R,S)\subseteq(V,W)$. A fundamental system of regular good pairs for $f$ is defined similarly. 

\begin{lemma}\label{good pair}
If there exists a fundamental system $I$ of good pairs (respectively, regular good pairs) for a map $f\from\Gamma\to\Delta$ of profinite graphs, then $f\from\Gamma\to\Delta$ is a covering (respectively, regular covering). 
\end{lemma}

\begin{proof}
The fundamental system $I$ is a directed set under the partial ordering given by the opposite of inclusion of pairs.  So, $\Gamma = \varprojlim \Gamma/R$ where $R$  runs through $\{R\mid (R,S)\in I\}$ and $\Delta = \varprojlim \Delta/S$ where $S$ runs through $\{S\mid (R,S)\in I\}$. And for each $(R,S), (R_1,S_1), (R_2,S_2) \in I, (R_2,S_2) \subseteq  (R_1,S_1),$ we have the two commutative diagrams as follows:
$$\begindc{\commdiag}[30]
\obj(0,0)[1]{$\Delta/S_2$}
\obj(6,0)[2]{$\Delta$}
\obj(2,0)[3]{$\Delta/S_1$}
\obj(8,0)[4]{$\Delta/S$}
\obj(0,2)[5]{$\Gamma/R_2$}
\obj(6,2)[6]{$\Gamma$}
\obj(2,2)[7]{$\Gamma/R_1$}
\obj(8,2)[8]{$\Gamma/R$}
\mor{1}{3}{$\psi_{S_1S_2}$}
\mor{2}{4}{$\psi_S$}
\mor{5}{1}{$f_{S_2R_2}$}
\mor{6}{2}{$f$}
\mor{7}{3}{$f_{S_1R_1}$}
\mor{8}{4}{$f_{SR}$}
\mor{5}{7}{$\phi_{R_1R_2}$}
\mor{6}{8}{$\phi_R$}
\enddc$$
where each $f_{SR}: \Gamma/R \rightarrow \Delta/S$ is locally bijective (respectively, a regular covering), and thus $f$ is a covering (resp. regular covering) of profinite graphs.
\end{proof}

\begin{theorem}\label{quo cov}
Let $G$ be a profinite group acting continuously and residually freely on a connected profinite graph $\Gamma$, without edge inversions. Then $\Gamma/G$ is a profinite graph and the orbit mapping $f\from \Gamma\to\Gamma/G$ is a regular covering of profinite graphs. Moreover, $G = \Aut(\Gamma, f)$ and they are isomorphic as profinite groups. 
\end{theorem}

\begin{proof}
Since $G$ is acting residually freely over the profinite graph $\Gamma$, by Definition~\ref{res free def}, there exists a fundamental system $I$ of $G$-invarient compatible cofinite entourages $R$ over $\Gamma$. Consider the natural orbit map of graphs $f:\Gamma \rightarrow \Gamma/G$ and the following diagram 
$$\begindc{\commdiag}[30]
\obj(0,0)[1]{$\Gamma/G$}
\obj(2,0)[2]{$(\Gamma/R)/G$}
\obj(0,2)[3]{$\Gamma$}
\obj(2,2)[4]{$\Gamma/R$}
\mor{1}{2}{$\zeta_R$}
\mor{3}{1}{$f$}
\mor{4}{2}{$f_R$}
\mor{3}{4}{$\phi_R$}
\enddc$$
where $f_{R}:\Gamma/R \rightarrow (\Gamma/R)/G$ is the quotient map of graphs corresponding to the group action of $G$ over $\Gamma/R$ defined by $g.R[x] = R[gx]$. Also $\phi_R$ is the natural quotient map from $\Gamma$ to $\Gamma/R$, and
$\zeta_{R}:\Gamma/G \rightarrow (\Gamma/R)/G$ is defined by $\zeta_{R}([x]) = f_{R}\phi_{R}(x),$ for all $[x]\in \Gamma/G$, where $[x]$ is the orbit of $x \in \Gamma$ in the quotient graph $\Gamma/G$. Also then the above diagram is a commutative diagram of continuous maps of graphs, where $\zeta_R$ is continuous since $f$ is a quotient map and $\zeta_{R}f = f_R\phi_R$ is a continuous map of graphs. Write $T_{R} = \zeta_{R}^{-1}\zeta_{R}$. 
and  $K =f^{-1}f$. Let's first discuss some important facts. 

\begin{enumerate}
\item[Claim (1)] $RK = KR$, for all $R \in I$.

The proof follows since each $R \in I$ is $G$-invarient.

\item[Claim (2)] $(f\times f)[R] = T_{R}$ for all $R \in I$.

\item[Claim (3)] $K$ is closed in $\Gamma \times \Gamma$

Proof: Consider the map $h\from G\times \Gamma \to \Gamma \times \Gamma$ defined by $h(g,x) = (x,gx)$. This is a continuous map from a profinite space  to profinite space. So, $K =$ Image of $h$ and hence $K$ is closed in $\Gamma \times \Gamma$.

\item[Claim(4)] $\Gamma/G = \varprojlim (\Gamma/R)/G$ where $R$ runs through $I$, the fundamental system of $G$ invariant compatible cofinite entourages over $\Gamma$. 

The proof follows by using the previous claims and since, for all $R\in I, G/N_R$ acts freely on $\Gamma/R$ with quotient graph $(\Gamma/R)/G$, the orbit map 
$f_R\from\Gamma/R \to (\Gamma/R)/G$ is a regular covering of finite graphs.

So, combining Claim (1) through Claim (4), we deduce that $\Gamma/G$ is a profinite graph and $f^{\prime} = (\varprojlim_{R\in I}f_R)\from \Gamma\to \Gamma/G$ is a regular covering. 

It turns out that $f^{\prime} = f$ and hence orbit mapping $f\from\Gamma\to\Gamma/G$ is a regular covering of profinite graphs.

Clearly, $G \subseteq \Aut(\Gamma,f)$ and by uniqueness of lifts $\Aut(\Gamma,f) \subseteq G$. Now note that the group action of $G$ over $\Gamma$ is uniformly continuous. So, since $\{N_R \mid R\in I\}$ form a fundamental system of entourages of $\Aut(\Gamma,f)$, it follows that the identity map from $G$ to $\Aut(\Gamma,f)$ is uniformly continuous and hence a homeomorphism as $G$ and $\Aut(\Gamma,f)$ are compact Hausdorff spaces. Hence, the result follows.

\end{enumerate}
\end{proof}

We next observe that, the converse of the lemma~\ref{good pair} is true for regular coverings of profinite graphs; we do not know whether or not the converse holds for non-regular coverings.
Hence, regular coverings can be characterized as follows.

\begin{theorem}\label{nec suf reg cov}
A map of profinite graphs $f\from\Gamma\to\Delta$ is a regular covering if and only if there exists a fundamental system of regular good pairs for $f$.
\end{theorem}

\begin{proof}
Since $f$ is a covering map of profinite graphs with $\Gamma$ profinitely connected, by Lemma~\ref{uniform equicont},  $G = \Aut(\Gamma,f)$ acts over $\Gamma$ uniformly equicontinuously and hence continuously. By Lemma~\ref{resi free act}, $G$ acts residually freely over $\Gamma$. So, by Theorem~\ref{quo cov}, the orbit map $q:\Gamma \rightarrow \Gamma/G$ is a regular covering of profinite graphs and each orbit map $f_R:\Gamma/R \rightarrow (\Gamma/R)/G$, induced by the action of $G$ over $\Gamma/R$, is a regular covering map of finite connected graphs. We show that there is an isomorphism of topological graphs between $\Gamma/G$  and $\Delta$. Define $\theta:\Gamma/G \rightarrow \Delta$ via $\theta([x]) = f(x)$ where $[x]$ is the orbit of $x$ under the orbit map $q:\Gamma \rightarrow \Gamma/G$. Then $\theta$ is a continuous, bijective map of graphs from a compact topological graph $\Gamma/G$ to a Hausdorff topological graph $\Delta,$ and so $\theta$ is an isomorphism of topological graphs. Now let $R_i = \phi_{i}^{-1}\phi_{i}, S_{j} = \psi_{j}^{-1}\psi_{j}$ be any two compatible cofinite entourages over $\Gamma$ and $\Delta$ respectively, where $R_i$ is $G$-invariant. Consider $S_k = \psi_{k}^{-1}\psi_{k}$ and suppose $S_k \subseteq S_i, S_j$. Suppose  $\phi_k^{-1}\phi_k = R_k$ is the $G$-invariant compatible cofinite entourage over $\Gamma$ with $R_k \subseteq R_i, R_j$. Also $(f\times f)(R_k) = T_{R_k}$, as in the Claim (2) in the proof of Theorem~\ref{quo cov} and $(f\times f)(R_k) \subseteq S_{k}$. This implies that $T_{R_k} \subseteq S_k$. So we have $(R_k, T_{R_{k}}) \subseteq (R_i,S_j)$, where $f_{R_k}\from\Gamma/R_k \to (\Gamma/R_k)/G$ is a regular covering map finite graphs. As in the proof of Theorem~\ref{quo cov} $(R,T_R)$ is a regular good pair for each $G$-invariant compatible cofinite entourage $R$ over $\Gamma$.  Hence it follows that if $f\from\Gamma\to\Delta$ is a regular covering of profinite graphs then there exists a fundamental system of regular good pairs for $f$. By considering Lemma~\ref{good pair} we obtain the converse part and thus the theorem.
\end{proof}

\begin{corollary}
 If $f\from \Gamma \to \Delta$ is a regular covering of profinite graphs then it is a uniform quotient map.
\end{corollary}

\begin{lemma}
Let $f\from\Gamma \to \Delta$ be a regular covering map of profinite connected graphs. Then for any closed subgroup $H$ of $G = Aut(\Gamma,f)$, we have the 
commutative diagram

$$\begindc{\commdiag}[30]
\obj(0,0)[1]{$\Gamma/H$}
\obj(3,0)[2]{$\Gamma/G\cong \Gamma/\!/G$}
\obj(3,3)[3]{$\Gamma$}
\mor{1}{2}{$f_H$}
\mor{3}{1}{$h$}
\mor{3}{2}{$f$}
\enddc$$
 where $f_H, h$ are a covering map and a regular covering map of profinite connected graphs respectively. Also $f_H$ is a regular covering if $H$ is a normal subgroup of $G$. 
\end{lemma}

\begin{proof}
 Each $G$-invariant compatible cofinite entourage $R$ over $\Gamma$ is also an $H$-invariant compatible cofinite entourage. Thus, as in Theorem~\ref{quo cov}, each $h_{R}\from \Gamma/R \to (\Gamma/R)/H$, the natural quotient map is a regular covering of finite path connected graphs. It follows that $h = \varprojlim_{R\in I}h_R$ is a regular covering  of profinite connected graphs.

As a consequence, we have a new commutative diagram of maps of finite path connected graphs, namely 
$$\begindc{\commdiag}[30]
\obj(3,3)[1]{$\Gamma/R$}
\obj(0,0)[2]{$(\Gamma/R)/H$}
\obj(6,0)[3]{$(\Gamma/R)/G$}
\mor{1}{2}{$h_R$}
\mor{2}{3}{$f_{HR}$}
\mor{1}{3}{$f_R$}
\enddc$$
 We have 

$$
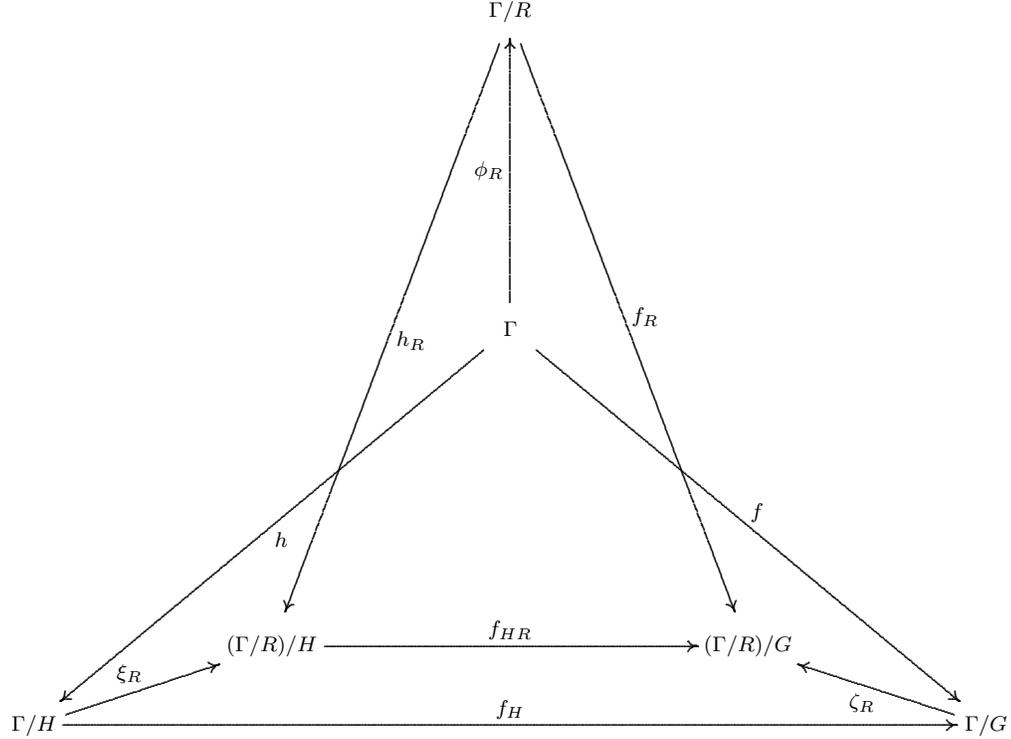
\begindc{\commdiag}[30]
\obj(0,-1)[1]{$\Gamma/H$}
\obj(6,4)[2]{$\Gamma$}
\obj(6,8)[3]{$\Gamma/R$}
\obj(3,0)[4]{$(\Gamma/R)/H$}
\obj(9,0)[5]{$(\Gamma/R)/G$}
\obj(12,-1)[6]{$\Gamma/G$}
\mor{2}{1}{$h$}
\mor{2}{6}{$f$}
\mor{1}{4}{$\xi_R$}
\mor{6}{5}{$\zeta_R$}
\mor{4}{5}{$f_{HR}$}
\mor{3}{5}{$f_R$}
\mor{3}{4}{$h_R$}
\mor{2}{3}{$\phi_R$}
\mor{1}{6}{$f_H$}
\enddc$$

Thus, for all $\delta\in\Gamma$, if $\delta = (\delta_R)_{R\in I}$, then $f_H(h(\delta)) = f_H(h((\delta_R)_{R\in I})) = f_H((h_R(\delta_R))_{R\in I}) = (f_{HR}((h_R(\delta_R)))_{R\in I} = (f_R(\delta_R))_{R\in I} = f((\delta_R)_{R\in I}) = f(\delta)$.

Now, let $H$ be a normal subgroup of $G$. In order to show $f_H$ is a regular covering we wish to show that each $f_{HR}$ is a regular covering of finite path connected graphs. This follows since $G$ acts on $(\Gamma/R)/H$ via covering transformation.

\end{proof}


\section{Profinite fundamental groups and the Universal coverings}

The goal of the next subsection is to construct a universal covering of any connected profinite graph $\Delta$.

We define a profinite analogue of the fundamental group of an abstract graph. Just as we have defined most notions here, we define this first in the special case of a finite discrete graph and then extend to an arbitrary profinite graph by taking the projective limit of the profinite fundamental groups of its finite discrete quotient graphs.

\begin{definition} \rm
Let $\Delta$ be a profinite graph and let $b$, be a vertex of $\Delta$. The {\it profinite fundamental group\/} $\widehat\pi_1(\Delta,b)$ of $\Delta$ based at $b$ is defined as follows:
\begin{enumerate}
\item[(a)] If $\Delta$ is finite, then $\widehat\pi_1(\Delta,b)$ is the profinite completion of the ordinary fundamental group $\pi_1(\Delta,b)$ with respect to the collection of all cofinite congruences over $\pi_1(\Delta,b)$. 
\item[(b)] In the general case, let $J$ be a fundamental system of compatible cofinite entourages of $\Delta$ and define $\widehat\pi_1(\Delta,b)=\varprojlim\widehat\pi_1(\Delta/S,S[b])$, where $S$ runs through~$J$.
\end{enumerate}
\end{definition}

\begin{remarks}
The structure, defined above is well defined by ~\cite{bH77}. 
\end{remarks}

Next we define homomorphisms on profinite fundamental groups induced by continuous map of profinite graphs.
\begin{definition} \label {Induced homomorphism}
\rm
 Let $f\from\Gamma \to \Delta$ be a continuous map of profinite graphs. First consider the case when $\Gamma$ and $\Delta$ are finite. Let $f$ denote the induced homomorphism on ordinary fundamental groups $f\from \pi_1(\Gamma,a) \to \  \pi_1(\Delta,f(a))$ for any vertex $a$ in $\Gamma$. Then, $f$ is continuous in the cofinite topologies in which every normal subgroup of finite index is open. Hence $f$ determines a continuous homomorphism of the profinite completions $f^{*}\from\widehat {\pi_1}(\Gamma,a) \to \  \widehat{\pi_1}(\Delta,f(a))$. Now we consider the general case where $\Gamma$ and $\Delta$ are any profinite graphs.
For any pair of compatible cofinite entourages $R$ and $S$ over $\Gamma$ and $\Delta$ respectively, let us call $(R,S)$  a {\it half good pair\/} relative to  $f$ if and only if 
$(f\times f)[R] \subseteq S$, that is if and only if there is well-defined map of finite graphs $f_{SR}\from \Gamma/R \to \Delta/S$ determined by $f$. Then, $J = \{(R,S)\mid (R,S)$ is a half good pair relative to $f\}$, endowed with reverse inclusion, i.e. $(R,S)\geq (T,Q)$ if and only if $R\subseteq T$ and $S\subseteq Q$, forms a directed set. Then $f = \varprojlim f_{SR}, ((R,S) \in J)$. So, for $(R,S), (T,Q) \in J$ with $T\subseteq R, Q\subseteq S$ we have the commutative diagram
$$\begindc{\commdiag}[30]
\obj(0,0)[1]{$\Delta/Q$}
\obj(4,0)[2]{$\Delta/S$}
\obj(0,4)[3]{$\Gamma/T$}
\obj(4,4)[4]{$\Gamma/R$}
\mor{1}{2}{$\psi_{SQ}$}
\mor{3}{1}{$f_{QT}$}
\mor{4}{2}{$f_{SR}$}
\mor{3}{4}{$\phi_{RT}$}
\enddc$$
where $\phi_{RT}, \psi_{SQ}$ are map of the quotient graphs defined in the natural way. So, the above commutative diagram leads to another commutative diagram of continuous homomorphism of profinite groups.
$$\begindc{\commdiag}[30]
\obj(0,0)[1]{$\widehat\pi_1(\Delta/Q,Q[f(a)])$}
\obj(4,0)[2]{$\widehat\pi_1(\Delta/S,S[f(a)])$}
\obj(0,4)[3]{$\widehat\pi_1(\Gamma/T,T[a]))$}
\obj(4,4)[4]{$\widehat\pi_1(\Gamma/R,R[a])$}
\mor{1}{2}{$\psi^{*}_{SQ}$}
\mor{3}{1}{${f^{*}_{QT}}$}
\mor{4}{2}{${f^{*}_{SR}}$}
\mor{3}{4}{$\phi^{*}_{RT}$}
\enddc$$

So, now we can define $f^{*} = \varprojlim f^{*}_{SR} \from \widehat {\pi_1}(\Gamma,a) \to \  \widehat{\pi_1}(\Delta,f(a))$

\end{definition}

\begin{lemma}
For any two continuous maps of profinite graphs $f$ from $(\Gamma,a)$ to $(\Delta,b)$ and $g$ from $(\Delta,b)$ to $(\Sigma,c)$ for any vertex $a \in V(\Gamma)$, $f(a) = b \in V(\Delta), g(b) = c \in V(\Sigma)$, consider the corresponding induced maps 
$f^{*}\from \widehat{\pi_1}(\Gamma,a)\to \widehat\pi_1(\Delta,b)$, $g^{*}\from \widehat{\pi_1}(\Delta,b)\to \widehat{\pi_1}(\Sigma,c)$  and the induced map
$(gf)^{*}\from \widehat{\pi_1}(\Gamma,a)\to \widehat{\pi_1}(\Sigma,c)$. Then $(gf)^{*} = g^{*}f^{*}$. 
\end{lemma}

\begin{lemma}
If $f\from \Gamma \to \Delta$ is an isomorphism of profinite graphs, then for any vertex $a$ in $\Gamma$ the induced homomorphism $f^{*}\from \widehat {\pi_1}(\Gamma,a) \to \  \widehat{\pi_1}(\Delta,f(a))$ is an isomorphism of profinite groups. 
\end{lemma}

From now we will denote the induced homomorphism $f^{*}\from \widehat {\pi_1}(\Gamma,a) \to \  \widehat{\pi_1}(\Delta,f(a))$ simply by $f\from \widehat {\pi_1}(\Gamma,a) \to \  \widehat{\pi_1}(\Delta,f(a))$
whenever $f\from\Gamma \to \Delta$ is a continuous map of profinite graphs and $a$ is a vertex in $\Gamma$.

\begin{lemma}
If $f\from \Gamma \to \Delta$ is a covering of profinite graphs, then for any vertex $a$ in $\Gamma$ the induced homomorphism $f\from \widehat {\pi_1}(\Gamma,a) \to \  \widehat{\pi_1}(\Delta,f(a))$ is a continuous monomorphism.
\end{lemma}
\begin{proof}
 The proof of the result follows by ~\cite{bH77}.
\end{proof}

\begin{lemma}\label{univ fund}
Let $\Gamma$ be a connected profinite graph. Let $a \in V(\Gamma)$. Then, $\widehat \pi_1(\Gamma,a) = 1$ if and only if for each compatible cofinite entourage $R$ over $\Gamma$ and subgroup $H$ of finite index in $\pi_1(\Gamma/R,R[a])$ there exists a compatible cofinite entourage $R^{\prime} \subseteq R$ such that 
$\phi_{RR^{\prime}}(\pi_1(\Gamma/R^{\prime},R^{\prime}[a]))\subseteq H$.
\end{lemma}

\begin{proof}
Let  us first assume that $\widehat \pi_1(\Gamma,a) = 1$. Then, let $H$ be any subgroup of finite index in $\pi_1(\Gamma/R,R[a])$ for some
compatible cofinite entourage $R$ over $\Gamma$. Without loss of generality we can assume that $H$ is a normal subgroup of finite index. Let the closure of $H$ in $\widehat\pi_1(\Gamma/R,R[a])$ be denoted by $\overline H$.  Also
$$
\bigcap_{R^{\prime\prime}\geq R} [\phi_{RR^{\prime\prime}}(\widehat\pi_1(\Gamma/R^{\prime\prime},R^{\prime\prime}[a]) )]\setminus \overline {H} = \emptyset
.
$$
Now each $\phi_{RR^{\prime\prime}}$ is a continuous map from a compact space to a Hausdorff space. So each $\phi_{RR^{\prime\prime}}(\widehat\pi_1(\Gamma/R^{\prime\prime},R^{\prime\prime}[a]))$ is a closed subset of $\widehat\pi_1(\Gamma/R,R[a])$. Thus by the finite intersubsection  property of compact topological space there exists finite number of compatible cofinite entourages $R_1, R_2, ... R_n$ over $\Gamma$, with $R_i\subseteq R,$ for all $i =1, 2, ..., n$, such that $\bigcap_{i =1}^n[\phi_{RR_i}(\widehat\pi_1(\Gamma/R_i,R_i[a]))] \setminus \overline {H} = \emptyset.$ So, $\bigcap _{i = 1}^n[\phi_{RR_i}(\widehat\pi_1(\Gamma/R_i,R_i[a]))] \subseteq \overline {H}$. Thus for the compatible cofinite entourage $R^{\prime} = \bigcap_{i =1}^nR_i\subseteq R$ over $\Gamma$, we have
 $\phi_{RR^{\prime}}(\pi_1(\Gamma/R^{\prime}, R^{\prime}[a]))
\subseteq H$.

For the converse part choose any $\gamma \in \widehat\pi_1(\Gamma, a)$, then, for any compatible cofinite entourage $R$ over $\Gamma$, consider 
$$
\phi_{R}(\gamma) \in\phi_{R}(\widehat\pi_1(\Gamma, a)) = \bigcap_{R^{\prime}\geq R}[\phi_{RR^{\prime}}(\widehat\pi_1(\Gamma/R^{\prime}, R^{\prime}[a]))]
$$
Now let $\overline{N}$ be a normal subgroup of finite indexed in $\widehat \pi_1(\Gamma/R, R[a])$.  Then, $\overline{N} \cap  \pi_1(\Gamma/R, R[a]) = N$, say, is a normal subgroup of finite index in $\pi_1(\Gamma/R, R[a])$. So, by the given condition, there exists a compatible cofinite entourage $R^{\prime}$ over $\Gamma$, with $R^{\prime} \subseteq R$ such that $\phi_{RR^{\prime}}(\pi_1(\Gamma/R^{\prime}, R^{\prime}[a]))\subseteq N$. Thus $\phi_{RR^{\prime}}(\widehat\pi_1(\Gamma/R^{\prime}, R^{\prime}[a]))\subseteq\overline{N}$. Hence by our earlier work we have, 
$$
\phi_{R}(\gamma) \in\bigcap_{R^{\prime}\geq R}[\phi_{RR^{\prime}}(\widehat\pi_1(\Gamma/R^{\prime}, R^{\prime}[a]))]\subseteq\phi_{RR^{\prime}}(\widehat\pi_1(\Gamma/R^{\prime}, R^{\prime}[a]))\subseteq\overline{N}
$$
Thus, $\phi_{R}(\gamma) \in \bigcap \overline{N}=1$. hence the result follows. 
\end{proof}

\begin{theorem}\label{univ exists lift}
Let $g\from\Tilde\Delta\to\Delta$ be a covering of connected profinite graphs and let $c$ be a vertex of $\Tilde\Delta$.
Let $\widehat\pi_1(\Tilde\Delta,c)=1, g(c) = b$. If $f\from\Gamma\to\Delta$ is a covering map of profinite graphs with $f(a) =b$ for some vertex $a \in V(\Gamma)$, then
 there exists a unique lift $h\from\Tilde\Delta\to\Gamma$ such that $fh = g, h(c) = a$. 
\end{theorem}
\begin{proof}
Proof of the result follows by the previous lemma and Theorem~\ref{Lifting Criterion}.
\end{proof}
\begin{lemma}\label{simply connect}
If $\Gamma$ is a connected profinite graph and $\widehat\pi_1(\Gamma, a_0) = 1$ for some vertex $a_0$, then $\widehat\pi_1(\Gamma, a) = 1$ for every vertex $a$.
\end{lemma}

\begin{proof}
We show that the property of the previous lemma holds for any vertex $a$ given that it holds for the vertex $a_0$. Let $R$ be a compatible cofinite entourage over $\Gamma$ and let $H$ be a normal subgroup of fnite index in $\pi_1(\Gamma/R, R[a])$. Choose a path $\gamma$ in $\Gamma/R$ from $R[a_0]$ to $R[a]$. Let $\gamma_{*}\from\pi_1(\Gamma/R, R[a_0]) \to \pi_1(\Gamma/R, R[a])$ denote the isomorphism given by $\gamma_{*}(\alpha) = \gamma^{-1}\alpha\gamma$. Then $\gamma_{*}^{-1}(H) = \gamma H\gamma^{-1}$ is a normal subgroup of finite index in $\pi_1(\Gamma/R, R[a_0])$. So, by the previous lemma, there exists a compatible cofinite entourage $R^{\prime} \subseteq R$ such that the image of 
$\phi_{RR^{\prime}} \from\pi_1(\Gamma/R^{\prime}, R^{\prime}[a_0]) \to \pi_1(\Gamma/R, R[a])$ lies in $\gamma H\gamma^{-1}$. Choose a path $\beta$ in $\Gamma/R^{\prime}$ from $R^{\prime}[a_0]$ to $R^{\prime}[a]$. We then have the following commutative diagram of continuous group homomorphisms.
$$\begindc{\commdiag}[30]
\obj(0,0)[1]{$\pi_1(\Gamma/R^{\prime},R^{\prime}[a])$}
\obj(4,0)[2]{$\pi_1(\Gamma/R,R[a])$}
\obj(0,4)[3]{$\pi_1(\Gamma/R^{\prime},R^{\prime}[a_0]))$}
\obj(4,4)[4]{$\pi_1(\Gamma/R,R[a_0])$}
\mor{1}{2}{$\phi_{RR^{\prime}}$}
\mor{3}{1}{${\beta_{*}}$}
\mor{4}{2}{$({\phi_{RR^{\prime}}\beta})_{*}$}
\mor{3}{4}{$\phi_{RR^{\prime}}$}
\enddc$$
So, since $\beta_{*}$ is an isomorphism we have that the image of the bottom homomorphism lies in the subgroup 
$({\phi_{RR^{\prime}}\beta})_{*}(\gamma H\gamma^{-1}) = 
({\phi_{RR^{\prime}}\beta})^{-1}\gamma H\gamma^{-1}({\phi_{RR^{\prime}}\beta}) = H$ as $\gamma^{-1}({\phi_{RR^{\prime}}\beta}) \in \pi_1(\Gamma/R,R[a])$ and $H$ is a normal subgroup of $\pi_{1}(\Gamma/R, R[a])$. Therefore, applying the previous lemma, we have $\widehat\pi_{1}(\Gamma,a) =1$.

\end{proof}

\begin{definition} \rm
Let $\Delta$ be a connected profinite graph. A connected profinite covering graph $(\widetilde\Delta,p)$ of $\Delta$ is called a {\it universal profinite covering graph\/} of $\Delta$ provided that the following universal property holds: if $(\Gamma,f)$ is a connected profinite covering graph of $\Delta$ and $c$, $a$ are vertices of $\widetilde\Delta$, $\Gamma$ with $p(c)=f(a)$, then there exists a homomorphism of coverings~$h$ from $(\widetilde\Delta,p)$ to $(\Gamma,f)$ such that $h(c)=a$.  
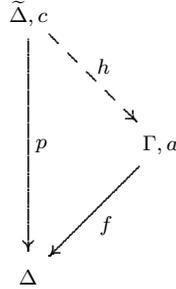
$$\begindc{\commdiag}[25]
\obj(0,4)[U]{$\widetilde\Delta,c$}
\obj(0,0)[D]{$\Delta$}
\obj(2,2)[G]{$\Gamma,a$}
\mor{U}{D}{$p$}
\mor{G}{D}{$f$}
\mor{U}{G}{$h$}[\atleft,\dashArrow]
\enddc$$
\end{definition}

\begin{remarks}
We have the following observations: 
\begin{enumerate}
\item The homomorphism $h$ from $(\widetilde\Delta,p)$ to $(\Gamma,f)$ with $h(c)=a$ for any given covering $(\Gamma,f)$ of $\Delta$ and vertices $c$, $a$ of $\widetilde\Delta$, $\Gamma$ respectively, with $p(c)=f(a)$ is unique by Lemma~\ref{l:Uniqueness of lifts}.
\item If $(\widetilde\Delta,p)$ and $(\widetilde\Delta^{\prime},p^{\prime})$ are two universal covering graphs of a connected profinite graph $\Delta$, then $(\widetilde\Delta,p)$ and $(\widetilde\Delta^{\prime},p^{\prime})$ are  isomorphic covering graphs. Hence, we can refer to {\it the\/} universal covering graph of~$\Delta$, once we know that it exists.
\item Applying Theorem ~\ref{univ exists lift} and Lemma ~\ref{simply connect} we can observe that a covering of connected profinite graphs $p\from \widetilde\Delta \to \Delta$ is a universal covering if $\widehat{\pi_1}(\widetilde\Delta, a) =1$ for any vertex $a$ in
$\widetilde\Delta$.
\end{enumerate}

\end{remarks}


\section{Existence of profinite coverings}

The goal of this section is to give a construction of the universal profinite covering graph for any connected profinite graph $\Delta$. In particular, we show that a universal covering of profinite graphs is a regular covering. Hence, the results in the sections on Regular Coverings and of Good pairs and residually free group actions can be applied and this leads to a characterization of all connected profinite covering graphs of a connected profinite graph $\Delta$ in terms of closed subgroups of the group $\Aut(\widetilde\Delta,p)$ of covering transformations of the universal profinite covering graph of $\Delta$. 

\begin{theorem}\label{exi univ cov}
Let $\Delta$ be a connected profinite graph. Then its universal profinite covering graph 
$(\widetilde\Delta,p)$ exists and it is regular. Furthermore, if $(\Sigma,g)$ is any profinite covering graph of $\Delta$ and $h$ is a homomorphism of coverings from $(\widetilde\Delta,p)$ to $(\Sigma,g)$, then the pair $(\widetilde\Delta,h)$ is a regular profinite covering graph of $\Sigma$.
\end{theorem}

\begin{proof}
 Without loss of generality we can assume that $\Delta =\varprojlim\Delta_{i}$ where $i$ runs through $I$ a directed set corresponding to a fundamental system of compatible cofinite entourages over $\Delta$ and each  $\Delta_{i} = \psi_{i}(\Delta) = \Delta/S$ with $\psi_{i}^{-1}\psi_i = S_i = S$ and $\psi_{i}$ being the canonical projection from $\Delta$ to $\Delta_{i}$. So, consider a typical $\Delta/S = \Delta_{i}$, and $S[b]  \in V(\Delta/S)$ where $b \in V(\Delta)$. Let $J = \{(S,N)\mid N$ a normal subgroup of finite index in $\pi_1(\Delta/S ,S[b]) , S = S_{i}$ for some $i \in I\}$. Then by path connected coverings there exists a finite connected graph $\Gamma_{SN}$, a vertex $a_{SN} \in V(\Gamma_{SN})$, and a regular covering of finite graphs $f_{SN}$ from $(\Gamma_{SN}, a_{SN})$ to $(\Delta/S ,S[b])$, so that $f_{SN}(\pi_1(\Gamma_{SN}, a_{SN})) = N$.
Let us now define an order over $J$ by $(S_j, N_j) \geq (S_{i}, N_i)$ if and only if $j \geq i$ in $I$ and $\psi_{ij}(N_{j})\subseteq N_{i}$, where $\psi_{ij}\from\pi_1(\Delta_j, b_j)\to\pi_1(\Delta_i,b_i)$ is the group homomorphism induced by the map of graphs $\psi_{ij}\from(\Delta_j, b_j)\to(\Delta_i,b_i)$, where $(\Delta_i, \psi_{ij})_{i,j \in I, j\geq i}$ is the corresponding inverse system and the vertex $b$ can also be viewed as $(b_i)_{i \in I}$. 

It turns out that that $J$ is a directed set with respect to the order $'\geq'$.

By virtue of the  lifting criterion of finite graphs and uniqueness of lifting we have an inverse system of finite discrete graphs $\{(\Gamma_{SN}, a_{SN})\}$ along with a family of continuous maps of graphs 
$$
\{\phi_{SNPM} \mid (S, N), (P, M)\in J,(P, M)\geq (S, N)\}
$$
So let us define $\widetilde\Delta = \varprojlim_{(S,N)\in J}(\Gamma_{SN}, a_{SN})$ and $p = \varprojlim_{{(S,N)}\in J}f_{SN}$. From the construction, $p\from(\widetilde\Delta,a)\to (\Delta,b)$ is a regular cover of profinite graphs. 

Now let $v\in V(\widetilde\Delta)$ where $v = (v_{SN})_{(S, N)\in J}$, we claim that $\widehat{\pi_1}(\widetilde\Delta, v) = 1$.

Proof of the claim: Let $\phi_{SN}\from (\widetilde\Delta, v)\to(\Gamma_{SN}, v_{SN})$ be the canonical projection map. Let $R_{SN} = \phi_{SN}^{-1}\phi_{SN}$. Then we can identify $(\widetilde\Delta/R_{SN},R_{SN}[v])$ with $\phi_{SN}((\widetilde\Delta, v))$, as a subgraph of $(\Gamma_{SN}, v_{SN})$ and $R_{SN}[v] = v_{SN}$. Hence we can view $K = \{R_{SN}\mid(S,N)\in J\}$ as a fundamental system of compatible cofinite entourages over $\widetilde\Delta$.  

Now for any $R_{SN}$ in $K, \widetilde\Delta/R_{SN}$ is a finite discrete graph and

$\phi_{SN}\from (\widetilde\Delta,v)\to (\widetilde\Delta/R_{SN},R_{SN}[v])$ is a continuous map of graphs. Thus for any $(P,M)\in J$ with $(P,M)\geq (S,N)$ and for the continuous map of graphs 

$\phi_{SNPM}\from (\Gamma_{PM},v_{PM})\ to (\widetilde\Delta/R_{SN},R_{SN}[v])$ the following diagram commutes. 
$$\begindc{\commdiag}[30]
\obj(0,0)[1]{$(\Gamma_{PM},v_{PM})$}
\obj(-4,4)[2]{$(\widetilde\Delta,v)$}
\obj(4,4)[3]{$(\widetilde\Delta/R_{SN},R_{SN}[v])$}
\mor{2}{1}{$\phi_{PM}$}
\mor{1}{3}{$\phi_{SNPM}$}
\mor{2}{3}{$\phi_{SN}$}
\enddc$$

Let $R_{PM} = \phi_{PM}^{-1}\phi_{PM}$. Hence, $R_{PM}\subseteq R_{SN}$ by the above commutative diagram. Now let $N_1$ be a normal subgroup of finite index in $\pi_1((\widetilde\Delta/R_{SN},R_{SN}[v]))$. Then $\phi_{SNPM}^{-1}(N_1)$ is a normal subgroup of finite index in $\pi_1(\Gamma_{PM},v_{PM})$. Hence $f_{PM}(\phi_{SNPM}^{-1}(N_1))$ is a subgroup of finite index in $\pi_1(\Delta/P, P[p(v)])$ as $f_{PM}(\pi_1(\Gamma_{PM},v_{PM}))$ is a subgroup of finite index in $\pi_1(\Delta/P,P[p(v)]) \cong\pi_1(\Delta/P,P[b])$ as $\Delta/P$ is path connected. Let $H$ a normal subgroup of finite index in $\pi_1(\Delta/P,P[p(v)])$ be such that $ H\subseteq f_{PM}(\phi_{SNPM}^{-1}(N_1))\bigcap M$. Hence there exists $(P,H)$ in $J$ such that $f_{PH}\from(\Gamma_{PH},v_{PH})\to(\Delta/P,P[p(v)])$ is a regular covering of finite graphs and $f_{PH}(\pi_1(\Gamma_{PH},v_{PH})) = H$. Then we have the following commutative diagram.
$$\begindc{\commdiag}[30]
\obj(0,0)[1]{$(\Gamma_{PM},v_{PM})$}
\obj(-3,3)[2]{$(\widetilde\Delta,v)$}
\obj(3,3)[3]{$(\widetilde\Delta/R_{SN},R_{SN}[v])$}
\obj(-5,0)[4]{$(\Gamma_{PH},v_{PH})$}
\obj(0,-4)[5]{$(\Delta/P,P[p(v)])$}
\mor{2}{1}{$\phi_{PM}$}
\mor{1}{3}{$\phi_{SNPM}$}
\mor{2}{3}{$\phi_{SN}$}
\mor{2}{4}{$\phi_{PH}$}
\mor{4}{1}{$\phi_{PMPH}$}
\mor{1}{5}{$f_{PM}$}
\mor{4}{5}{$f_{PH}$}
\enddc$$
Thus 
$$
f_{PM}(\phi_{PMPH}(\pi_1((\Gamma_{PH},v_{PH})))) \subseteq H
$$
Hence 
$$
\phi_{PMPH}(\pi_1((\Gamma_{PH},v_{PH}))) \subseteq f_{PM}^{-1}(H)\subseteq \phi_{SNPM}^{-1}(N_1)
$$
Calling $\phi_{PH}^{-1}\phi_{PH} = R_{PH}$, as earlier, we have now $R_{PH} \subseteq R_{PM} \subseteq R_{SN}$ and

$$
\phi_{SNPH}(\pi_1(\widetilde\Delta/R_{PH},R_{PH}[v])) = \phi_{SNPM}(\phi_{PMPH}(\pi_1(\widetilde\Delta/R_{PH},R_{PH}[v])))
$$

$$
\subseteq \phi_{SNPM}(\phi_{PMPH}(\pi_1((\Gamma_{PH},v_{PH})))) \subseteq \phi_{SNPM}((\phi_{SNPM}^{-1}(N_1)) \subseteq N_1
$$
Hence, by the Lemma~\ref{univ fund}, $\widehat{\pi_1}(\widetilde\Delta, v) = 1$. So, by a remark at the end of section 6 we have $\Tilde\Delta$ is the universal covering of $\Delta$.

Let $g\from \Gamma\to \Delta$ be any covering of profinite connected graphs. Then, by section 5, $\Gamma = \Tilde\Delta/H$ for some closed subgroup $H$ of $G = \Aut(\Tilde\Delta,p)$. So, as in section 5, since $H$ acts residually freely over $\Tilde\Delta$ we can view $g\from \Gamma \to \Delta$ as the orbit map $h\from\Tilde\Delta \to \Tilde\Delta/H$ which is a regular covering.
Hence the result follows.  
\end{proof}

Thus in the light of {Theorem}~\ref{univ exists lift}, Lemma $6.8$, {Theorem}~\ref{exi univ cov}, we have the following theorem
\begin{theorem}
Let $\Delta$ be a connected profinite graph. Then a covering map $p\from\widetilde{\Delta}\to\Delta$ is the universal covering of $\Delta$ if and only if $\widehat\pi_1(\widetilde\Delta,v) =1$ for any vertex $v \in V(\widetilde\Delta)$
\end{theorem}

\begin{theorem}
For any covering $f\from\Gamma\to\Delta$ of connected profinite graphs, $\Gamma$ can be viewed as $\widetilde{\Delta}/H$ where $p\from \widetilde\Delta \to \Delta$ is the universal covering map and $H$ is a closed subgroup of $G = Aut(\widetilde\Delta,p)$.
\end{theorem}

\begin{proof}
Let $f\from\Gamma \to \Delta$ be a covering map of profinite connected graphs. Then, as in the proof of Theorem~\ref{exi univ cov}, there exists a unique lift 
$h\from\widetilde\Delta \to \Gamma$, where $h$ is a regular covering map of profinite connected graphs and $fh = p$. Then, as in the proof of Theorem~\ref{nec suf reg cov}, $\Gamma$ is isomorphic to $\widetilde\Delta /H$, where $H = \Aut(\widetilde{\Delta},h)$. 

Now it turns that $H$ is a subgroup of $G$. Since $H$ is a subgroup of $G$ all $G$-invariant compatible cofinite entourages $R$ over $\Gamma$ are also $H$-invariant. It follows that $H$ is a uniform subspace of $G$. So, by Theorem~\ref{quo cov}, $H$ is compact and $G$ is Hausdorff. Hence $H$ is closed in $G$. 
\end{proof}

\renewcommand{\bibname}{References}
\bibliographystyle{plain}

\end{document}